\documentclass[12pt,psamsfonts]{amsart}
\usepackage{amsmath,amsthm,amsfonts,amssymb,pstricks}

\usepackage{eucal}
\newcommand{\1}{{1 \hspace{-0.35em} {\rm 1}}}

\newcommand{\tr}{{\rm tr}}

 \DeclareMathOperator{\Sp}{Sp}
\DeclareMathOperator{\End}{End} 
\DeclareMathOperator{\Hom}{Hom} \DeclareMathOperator{\im}{im}
\DeclareMathOperator{\Irr}{Irr} 
 \DeclareMathOperator{\Id}{Id}
 \DeclareMathOperator{\diag}{diag}\DeclareMathOperator{\Tr}{Tr}

\newcommand{\C}{\mathbb C}

\newtheorem{thm}{Theorem}[subsection]

\theoremstyle{definition}
\newtheorem{defn}[thm]{Definition}
\theoremstyle{remark}
\newtheorem{rmk}[thm]{Remark}

\begin{document}
\title[From ribbon categories to gYB-operators and link invariants]
{From ribbon categories to generalized Yang-Baxter operators and link invariants (after Kitaev and Wang)}

\author{Seung-moon Hong}
\email{seungmoon.hong@utoledo.edu}
\address{Department of Mathematics and Statistics\\
    The University of Toledo \\
    Toledo, OH 43606\\
    U.S.A.}

\thanks{ We thank Zhenghan Wang and Eric C. Rowell for their encouragement and useful discussions.}

\begin{abstract}
We consider two approaches to isotopy invariants of oriented links: one from ribbon categories and the other from generalized Yang-Baxter operators with appropriate enhancements. The generalized Yang-Baxter operators we consider are obtained from so-called gYBE objects following a procedure of Kitaev and Wang.  We show that the enhancement of these generalized Yang-Baxter operators is canonically related to the twist structure in ribbon categories from which the operators are produced.  If a generalized Yang-Baxter operator is obtained from a ribbon category, it is reasonable to expect that two approaches would result in the same invariant.  We prove that indeed the two link invariants are the same after normalizations. As examples, we study a new family of generalized Yang-Baxter operators which is obtained from the ribbon fusion categories $SO(N)_2$, where $N$ is an odd integer.  These operators are given by $8\times 8$ matrices with the parameter $N$ and the link invariants are specializations of the two-variable Kauffman polynomial invariant $F$.

\end{abstract}

\maketitle
\section{Introduction}

There are two two-variable generalizations of the Jones polynomial: the HOMFLY-PT polynomial $P$ and the Kauffman polynomial $F$.  Turaev constructed these polynomial invariants of links $P$ and $F$ using Yang-Baxter operators \cite{Tu}. Recently \emph{generalized} Yang-Baxter operators were proposed \cite{RZWG}. Since then only a few generalized Yang-Baxter operators have been discovered. Their braid group representations and applications to link invariants were studied in \cite{GHR, Ch, Ho}. We obtain explicitly another family of generalized Yang-Baxter operators from ribbon categories using the method in \cite{KW} and study their corresponding link invariants.

Each object in a ribbon category gives rise to an isotopy invariant of \emph{framed} oriented links (or oriented ribbon graphs). Twist structure in ribbon categories is an obstruction to the construction of an isotopy invariant of (non-framed) oriented links.  Usually one needs to introduce a normalization factor to compensate for the twist structure. Yang-Baxter operators have a similar obstruction and thus need to be enhanced to produce a link invariant. The enhancement of Yang-Baxter operators was introduced in \cite{Tu} and is essentially amount to a normalization to cancel out the twist effect (it corresponds to the second Markov move $\xi \rightarrow \xi \sigma^{\pm 1}_n$ for braids $\xi \in B_n$). This enhancement method is extended to generalized Yang-Baxter operators \cite{Ho}. If we obtain a generalized Yang-Baxter operator from a ribbon category and define an invariant of oriented links from the operator via such an enhancement, it is reasonable to expect that the resulting link invariant is related to the one from the ribbon category.

We demonstrate a method of constructing generalized Yang-Baxter operators from certain objects in ribbon fusion categories \cite{KW}. We show that one can enhance the operators obtained in this way using the twist data in the ribbon categories which we start with. This enhancement is canonical in the sense that it is always possible to enhance the operators and the resulting link invariant is the same as the one directly obtained from the category. As an example we construct a family of $(2, 3, 1)$-generalized Yang-Baxter operators from the unitary modular categories $SO(N)_2$ and define isotopy invariants of oriented links via the enhancement method. Furthermore we show that the invariants are specializations of the Kauffman polynomial invariant $F$.

The contents of this paper are as follows. In Section \ref{section-pre} we recall some prerequisite materials. We list a few necessary definitions on generalized Yang-Baxter operators and their corresponding link invariants. Also we present a convenient choice of trivalent basis morphisms in ribbon categories. Such a choice of trivalent basis makes it easy to calculate graphical expressions which appear in later sections. In Section \ref{section-gYB} we demonstrate how to obtain generalized Yang-Baxter operators from certain objects in ribbon fusion categories following \cite{KW}. These operators come with braid group representations in a natural way.  Based on such a representation one obtains link invariants essentially by considering traces of the images under the representations. In Section \ref{section-inv} we define invariants of oriented links associated with the generalized Yang-Baxter operators and show the equivalence of the resulting invariants with those coming directly from the categories. In Section \ref{section-examples} we present a new family of generalized Yang-Baxter operators which is obtained from the categories $SO(N)_2$. We show that the invariants are specializations of the Kauffman polynomial invariant $F$.

\section{Preliminaries}\label{section-pre}

\subsection{Generalized Yang-Baxter operators and link invariants}

Most of the material in this subsection can be found in \cite{Ho}.

\begin{defn}
An isomorphism $R:V^{\otimes k} \rightarrow V^{\otimes k}$ is called a \textbf{generalized Yang-Baxter operator} (briefly, a gYB-operator) of type $(d,k,m)$ if it satisfies the following generalized Yang-Baxter equation and far-commutativity:

$$(R\otimes I_m)\circ(I_m \otimes R)\circ(R \otimes I_m)=(I_m \otimes R)\circ(R \otimes I_m)\circ(I_m \otimes R);$$
$$(R\otimes I^{\otimes (j-2)}_m)\circ(I^{\otimes (j-2)}_m\otimes R)=(I^{\otimes (j-2)}_m\otimes R)\circ(R\otimes I^{\otimes (j-2)}_m)\:\: \text{for}\:\: j \geq 4$$

\noindent where $d=\dim(V)$ and $I_m = \Id_{V^{\otimes m}}$.
\end{defn}

Note that a $(d,2,1)$ type gYB-operator is the ordinary YB-operator on $V$ of dimension $d$. In this case,  far-commutativity is automatic as $R$ acts on disjoint tensor factors. However it is not so in general.

Each gYB-operator $R$ gives rise to a representation of braid group $B_n \rightarrow \End(V^{\otimes k+m(n-2)})$ via $\sigma_i \mapsto R_i=I_m^{\otimes i-1}\otimes R \otimes I_m^{\otimes n-i-1}$. We denote this representation by $\rho^R_n$ and its image by $\im(\rho^R_n)$.

It is well known that any oriented link can be obtained by closing a braid and two braids produce isotopic links if and only if these braids are related by a finite sequence of Markov moves $\xi \mapsto \eta^{-1}\xi\eta, \xi \mapsto \xi\sigma_n^{\pm1}$ where $\xi, \eta \in B_n$. One may expect the following process as a way of constructing an invariant of oriented links: for each oriented link choose a braid whose closure is isotopic to the link, apply representation $\rho^R_n$, and then compute the trace of the image. However this process does not give us an invariant of oriented links because it does not respect the second Markov moves in general. This is why we need to enhance gYB-operators.

\begin{defn}\label{def-EgYB}
 An \textbf{enhanced generalized Yang-Baxter operator} (EgYB-operator) is a collection \{a gYB-operator $R:V^{\otimes k} \rightarrow V^{\otimes k}$, $\mu :V \rightarrow V $, invertible elements $\alpha, \beta$ of $\mathbb{C}$ \} which satisfies the following conditions for all $n$:

(i) The endomorphism $\mu^{\otimes k}  : V^{\otimes k} \rightarrow V^{\otimes k}$ commutes with $R$;

(ii) $\mu^{\otimes m(n-1)} \otimes (\Sp_{k,m}(R\circ \mu^{\otimes k})-\alpha\beta\mu^{\otimes k-m}) \in \im(\rho^R_n)^\bot$;\\ $\mu^{\otimes m(n-1)} \otimes (\Sp_{k,m}(R^{-1}\circ \mu^{\otimes k})-\alpha^{-1}\beta\mu^{\otimes k-m}) \in \im(\rho^R_n)^\bot$.

\end{defn}

\noindent Here $\Sp_{k,m}:\End(V^{\otimes k})\rightarrow \End(V^{\otimes k-m}), m < k,$ is the operator trace map which preserves trace. For example, if $f \in \End(V^{\otimes 3})$ then $\Sp_{3,1}(f)\in \End(V^{\otimes 2})$ is defined by $\Sp_{3,1}(f)(v_{i_1}\otimes v_{i_2})=\sum_{j}f^{j_1,j_2,j}_{i_1,i_2,j}v_{j_1}\otimes v_{j_2}$ using multiindex form. Orthogonality conditions (ii) are with respect to the trace inner product $\langle f,g\rangle=\tr(f^*\circ g)$ (see \cite{Ho} for details). For each EgYB-operator $S=\{R, \mu, \alpha, \beta \}$ we define an invariant of oriented links as follows:

$$T_S(L)=\alpha^{-w(D)}\beta^{-n} \tr(\rho^R_n(\xi)\circ \mu^{\otimes k+m(n-2)})$$

\noindent where $L$ is an oriented link, $D$ is a link diagram of $L$, $\xi \in B_n$ is a braid whose closure is isotopic to $L$, and $w(D)$ is the writhe.
It is known that the invariant is projectively multiplicative on disjoint union of links, $T_S(L_1\bigsqcup L_2)=\tr(\mu)^{2m-k} T_S(L_1)\cdot T_S(L_2)$, and thus invariant $\hat{T}_S := \tr(\mu)^{2m-k} T_S$ is multiplicative \cite{Ho}.

\subsection{Notations and Conventions for ribbon categories}

%We follow notations and conventions in \cite{ENO}

Let $\mathcal{C}$ be a ribbon fusion category and $\Irr (\mathcal{C})=\{X_i \}$ be a set of isomorphism classes of simple objects. $N^{k}_{i, j}$ denotes the multiplicity of $X_k$ in the tensor product decomposition of $X_i \otimes X_j$. If $N^{k}_{i, j}$ is nonzero, we may choose a basis $\{v_1, v_2, \ldots, v_{N^{k}_{i, j}} \}$ of $\Hom (X_i \otimes X_j,X_k)$ and dual basis $\{v'_1, v'_2, \ldots, v'_{N^{k}_{i, j}} \}$ of $\Hom (X_k, X_i \otimes X_j)$ so that the composite $v_l \circ v'_m$ is equal to zero if $l \neq m$ and nonzero multiple of $\Id_{X_k}$ if $l=m$. If $\mathcal{C}$ is pseudo-unitary, then all quantum dimensions $\dim(X_i)$ are positive and we may choose basis so that
$v_l \circ v'_m = \delta_{l,m}\frac{\sqrt{\dim(X_i) \dim(X_j)        }}{\sqrt{   \dim(X_k)    }} \Id_{X_k}$. Then these trivalent basis morphisms satisfy

 \begin{equation}\label{trivalent}
\begin{array}{c} \Id_{X_i \otimes X_j}= \sum \frac{\sqrt{   \dim(X_k)     }}{\sqrt{ \dim(X_i) \dim(X_j)  }}v'_l \circ v_l \\
\Tr(v_l \circ v'_l)=\sqrt{ \dim(X_i) \dim(X_j) \dim(X_k)         }
\end{array}
  \end{equation}

 \noindent where sum is taken over all simple objects $X_k$ with nonzero $N^{k}_{i, j}$, $l \in \{1, 2, \ldots,  N^{k}_{i, j}\}$, and $\Tr$ is Markov trace in the category. This choice is convenient for graphical calculus which we will use in later sections. Graphical expressions of these equalities are given in Figure \ref{basis}. For the case of non-unitary ribbon categories, we choose instead $\nu_i$ such that $\nu_i^2=\dim(X_i)$ and choose orthogonal basis in a similar way. However we will use the above square root notation as the main examples are pseudo-unitary.

\vspace{1cm}
\begin{figure}[h]\psset{unit=5mm}

$\begin{pspicture}(0,0)(2,4)
\psline(1,0)(1,1)
\psline(1,3)(1,4)
\psbezier(1,1)(0,1.8)(0,2.2)(1,3)
\psbezier(1,1)(2,1.8)(2,2.2)(1,3)
\put(1,0){\tiny $X_k$}
\put(1,3.5){\tiny$X_k$}
\put(-0.5,2){\tiny$X_i$}
\put(1.8,2){\tiny$X_i$}
\put(0.3,0.6){\tiny$v'_m$}
\put(0.4,3){\tiny$v_l$}
\end{pspicture}$
$\begin{array}{c} =\: \delta_{l,m}\frac{\sqrt{d_i d_j}}{\sqrt{d_k}}\\ \\ \\ \\  \end{array}$
$\begin{pspicture}(0,0)(2,4)
\psline(0,0)(0,4)
\put(0,0){\tiny$X_k$}
\put(1.2,0){,}
\end{pspicture}$
\hspace{0.1cm}$\begin{pspicture}(0,0)(1,4)
\psline(0,0)(0,4)
\psline(1,0)(1,4)
\put(0,0){\tiny$X_i$}
\put(1,0){\tiny$X_j$}
\end{pspicture}$
$\begin{array}{c} = \: \sum \frac{\sqrt{d_k}}{\sqrt{d_i d_j}} \\ \\ \\ \\  \end{array}$
\hspace{-0.3cm}$\begin{pspicture}(0,0)(2,4)
\psline(0,0)(1,1)
\psline(2,0)(1,1)
\psline(0,4)(1,3)
\psline(2,4)(1,3)
\psline(1,1)(1,3)
\put(0.3,0){\tiny$X_i$}
\put(2,0){\tiny$X_j$}
\put(0.2,3.8){\tiny$X_i$}
\put(2,3.8){\tiny$X_j$}
\put(1,2){\tiny$X_k$}
\put(0.3,1){\tiny$v_l$}
\put(0.3,2.6){\tiny$v'_l$}
\put(3.3,0){,}
\end{pspicture}$
\hspace{0.9cm}$\begin{pspicture}(0,0)(3,4)
\psline(3,1)(3,3)
\psbezier(1,3)(1,4)(3,4)(3,3)
\psbezier(1,1)(1,0)(3,0)(3,1)
\psbezier(1,1)(0,1.8)(0,2.2)(1,3)
\psbezier(1,1)(2,1.8)(2,2.2)(1,3)
\put(1,-0.2){\tiny$X_k$}
\put(1,3.8){\tiny$X_k$}
\put(-0.5,2){\tiny$X_i$}
\put(1.8,2){\tiny$X_i$}
\put(0.3,0.6){\tiny$v'_l$}
\put(0.4,3){\tiny$v_l$}
\end{pspicture}$
$\begin{array}{c} = \:\:\sqrt{d_i d_j d_k}\\ \\ \\ \\  \end{array}$

\caption{\small In the equalities $d_i$ denotes $\dim(X_i)$. Trivalent vertices correspond to basis morphisms in $\Hom$-spaces and diagrams should be read from bottom to top. Once we choose dual basis $v'$ of $v$ satisfying the first condition, the second and third equalities easily follow.}
\label{basis}
\end{figure}

For any objects $X, Y, Z, W$ in a ribbon category $\mathcal{C}$, associativity constraints are given by $F^{ X, Y, Z    }_{ W }: \Hom( W, (X\otimes Y)\otimes Z  ) \rightarrow \Hom( W, X \otimes (Y\otimes Z)  )$ subject to the pentagon axiom. Braidings are given by $c_{X, Y}: X\otimes Y \rightarrow Y \otimes X$ subject to the hexagon axiom. We denote by $R^{X,Y}_{Z}$ the eigenvalue of braiding $c_{X, Y}$ on a subobject $Z$ of $X \otimes Y$. Twist coefficient of an object $X$ is denoted by $\theta_X$. See Figure \ref{FRtheta} for graphical notations.

\vspace{1cm}
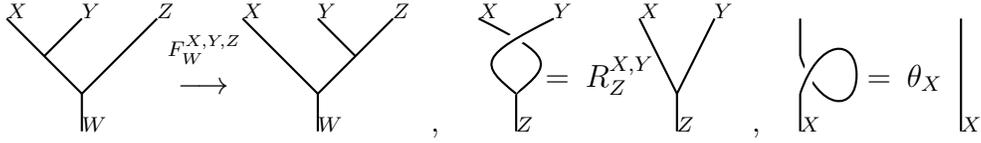
\begin{figure}[h]\psset{unit=5mm}

$\begin{pspicture}(0,0)(4,3)
\psline(2,0)(2,1)\psline(2,1)(0,3)\psline(1,2)(2,3)\psline(2,1)(4,3)
\put(2,0){\tiny $ W $}\put(0,3){\tiny $ X $}\put(2,3){\tiny $ Y $}\put(4,3){\tiny $ Z $}
\end{pspicture}$
\hspace{1cm}$\begin{pspicture}(0,0)(4,3)
\psline(2,0)(2,1)\psline(2,1)(0,3)\psline(3,2)(2,3)\psline(2,1)(4,3)
\put(2,0){\tiny $ W $}\put(0,3){\tiny $ X $}\put(2,3){\tiny $ Y $}\put(4,3){\tiny $ Z $}
\put(-2,1){ $ \longrightarrow $}\put(-2,2){\tiny $F^{X,Y,Z}_{W}$}
\put(5,0){,}
\end{pspicture}$
\hspace{1cm}$\begin{pspicture}(0,0)(2,3)
\psline(1,0)(1,1)
\psbezier(1,1)(2,1.8)(2,2)(0,3)\psbezier[linecolor=white,linewidth=5pt](1,1)(0,1.8)(0,2)(2,3)
\psbezier(1,1)(0,1.8)(0,2)(2,3)
\put(0,3){\tiny $ X $}\put(2,3){\tiny $ Y $}\put(1,0){\tiny $ Z $}
\end{pspicture}$
\hspace{1cm}$\begin{pspicture}(0,0)(2,3)
\psline(1,0)(1,1)\psline(1,1)(0,3)\psline(1,1)(2,3)
\put(0,3){\tiny $ X $}\put(2,3){\tiny $ Y $}\put(1,0){\tiny $ Z $}
\put(-2.5,1.2){$= \: R^{X,Y}_{Z}$}
\put(3,0){,}
\end{pspicture}$
\hspace{1cm}$\begin{pspicture}(0,0)(2.5,3)
\psline(0,0)(0,1)\psline(0,2)(0,3)
\psbezier(0,2)(0.5,0.5)(1.5,0.5)(1.5,1.5)\psline[linecolor=white,linewidth=5pt](0.1,1.2)(0.3,1.7)
\psbezier(0,1)(0.5,2.5)(1.5,2.5)(1.5,1.5)
\put(0,0){\tiny $ X $}
\end{pspicture}$
\hspace{0.5cm}$\begin{pspicture}(0,0)(2,3)
\psline(0.5,0)(0.5,3)
\put(0.5,0){\tiny $X$}
\put(-2,1.2){$= \: \theta_{X}$}
\end{pspicture}$

\caption{\small Graphical notations of associativity $F$, braiding $R$, and twist $\theta$}
\label{FRtheta}
\end{figure}

\section{Generalized Yang-Baxter operators from ribbon categories }\label{section-gYB}

We assume all multiplicities $N^{k}_{i, j}$ are either $0$ or $1$ for simplicity. However all results can be extended to the general case. The following definition is due to \cite{KW}.

\begin{defn}

Let $\mathcal{C}$ be a ribbon fusion category and $X$ be an object. We call $X$ a $(d,3,1)$-gYBE object with respect to a set $\mathcal{L}=\{X_i \}_{i \in I} \subset \Irr (\mathcal{C})$ if the set $\mathcal{L}$ consists of $d$ simple objects and $X \otimes X_j \cong \oplus_{i \in I}X_i$ for all $j \in I$.

\end{defn}

We will consider only a simple object $X$ for gYBE object as otherwise twist $\theta_X$ would be a matrix rather than a scalar. For each $(d,3,1)$-gYBE object $X$ with respect to a set $\mathcal{L}=\{X_i \}_{i \in I}$ one obtains a gYB-operator $R_{X, \mathcal{L}}$ of type $(d,3,1)$ as follows. Let $V^{(n)}_{X, \mathcal{L}}=\oplus_{i_1, i_{n+1} \in I} \Hom \left(X_{i_{n+1}}, X_{i_1} \otimes X^{\otimes n}   \right)$. The vector space $V^{(n)}_{X, \mathcal{L}}$ is $d^{n+1}$ dimensional and has an orthogonal basis consisting of all admissible trees as shown in Figure \ref{tree} in graphical notation. $(d, 3, 1)$-gYB-operator $R_{X, \mathcal{L}}$ is obtained by applying of braiding $c_{X, X}$ onto the two strands labeled by $X$ in $V^{(2)}_{X, \mathcal{L}}$. Braid group representation $\rho^{R_{X, \mathcal{L}}}_{n}: B_n \rightarrow \End(V^{(n)}_{X, \mathcal{L}})$ is defined by mapping a braid generator $\sigma_i$ to the action of $c_{X, X}$ on $i$-th and $(i+1)$-th strands labeled by $X$. If we identify $V^{(n)}_{X, \mathcal{L}} $ with $V^{\otimes n+1}$ for a $d$-dimensional vector space $V$ with a basis $\{ v_{i} \}_{i \in I}$, the set of all admissible trees shown in Figure \ref{tree} corresponds to an orthogonal basis $\{ v_{i_1}\otimes v_{i_2} \otimes \cdots \otimes v_{i_{n+1}} | i_1,i_2, \ldots , i_{n+1} \in I \}$ of $V^{\otimes n+1}$ and $\rho^{R_{X, \mathcal{L}}}_{n} (\sigma_{i})= \Id^{\otimes i-1}_{V}\otimes R_{X, \mathcal{L}} \otimes \Id^{\otimes n-i-1}_{V}$. The following theorem is due to \cite{KW}.

\vspace{1cm}
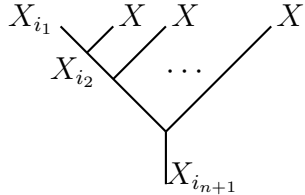
\begin{figure}[h]\psset{unit=7mm}

$\begin{pspicture}(0,0)(4,3)
\psline(2,0)(2,1)
\psline(2,1)(0,3)
\psline(0.5,2.5)(1,3)
\psline(1,2)(2,3)
\psline(2,1)(4,3)
\put(-1,3){$ X_{i_1}$}
\put(-0.2,2){$X_{i_2}$}
\put(2,0){$X_{i_{n+1}}$}
\put(2,2){$\cdots$}
\put(1.1,3){$X$}
\put(2.1,3){$X$}
\put(4.1,3){$X$}
\end{pspicture}$

\caption{\small All admissible labelings $X_{i_{1}}, X_{i_{2}}, \ldots, X_{i_{n+1}} \in \mathcal{L}$ of $n+1$ edges other than $X$-labeled edges form an orthogonal basis of $V^{(n)}_{X, \mathcal{L}}$ and thus $V^{(n)}_{X, \mathcal{L}}$ is $d^{n+1}$ dimensional. }
\label{tree}
\end{figure}

\begin{thm}\label{KWthm}
\begin{enumerate}
\item Any $(d, 3, 1)$-gYBE object $X$ leads to a solution of the $(d, 3, 1)$-gYBE which satisfies the far commutativity.
\item The number of distinct eigenvalues of each braid generator is less than or equal to  $\sum_{Y_k \in Irr{\mathcal{C}}}\dim \left( \Hom (Y_k, X\otimes X) \right)$. The resulting braid group representation is always reducible.
\item $\dim (X)=d$, where $d$ is the cardinality of the index set $I$ for $\mathcal{L}$.
\end{enumerate}
\end{thm}

\begin{rmk}\label{rmk-1}
\begin{enumerate}
\item $(d, 3, 1)$-gYB-operator $R_{X, \mathcal{L}}$ is given by a composite $\oplus_{i,j \in I} \left(F^{X_i,X,X}_{X_j}\right)^{-1} \circ \diag(R^{X,X}_{X_k}) \circ F^{X_i,X,X}_{X_j}$ where $R^{X,X}_{X_k}$ is the braiding eigenvalue of $c_{X,X}$ on a subobject $X_k$ of $X \otimes X$ for which $N^{X_i,X_k}_{X_j}$ is nonzero.

\item Using the orthogonal basis $\{ v_{i_1}\otimes v_{i_2} \otimes v_{i_3} | i_1, i_2, i_3 \in I\}$ of $V^{\otimes 3}$, $R_{X, \mathcal{L}}$ determines the multiindexed matrix $\left( (R_{X, \mathcal{L}})^{j_1,j_2,j_3}_{i_1,i_2,i_3} \right)$ by the equation $R_{X, \mathcal{L}}(v_{i_1}\otimes v_{i_2} \otimes v_{i_3})=\sum_{j_1, j_2, j_3 \in I} (R_{X, \mathcal{L}})^{j_1,j_2,j_3}_{i_1,i_2,i_3}  v_{j_1}\otimes v_{j_2} \otimes v_{j_3}$. From the construction, it is easy to see that $R_{X, \mathcal{L}} \in \End(V^{\otimes 3})$ acts diagonally on the first and third tensor factors. That is, $(R_{X, \mathcal{L}})^{j_1,j_2,j_3}_{i_1,i_2,i_3}=0$ unless $i_1=j_1$ and $i_3=j_3$. As a result, $\rho^{R_{X, \mathcal{L}}}_{n}(\xi)$ acts diagonally on the last tensor factor of $V^{\otimes k+m(n-2)}$ for any $\xi \in B_n$. And thus if $f \in \End(V^{\otimes k+m(n-2)})$ acts off-diagonally on the last tensor factor, then $f$ belongs to $\im(\rho^R_n)^\bot $.

\end{enumerate}
\end{rmk}

\section{Invariants of links}\label{section-inv}

Every object in a ribbon category $\mathcal{C}$ gives rise to an isotopy invariant of oriented \emph{framed} links (or oriented ribbon graphs). We will denote this invariant by $\langle L \rangle^{fr}_{\mathcal{C},X}$ for an oriented framed link $L$. It is preserved under Reidemeister moves $\Omega 2$ and $\Omega 3$ but not under $\Omega 1$. Rather it is preserved under $\Omega 1'$ with the choice of $\alpha = \theta_X$ (see Figure \ref{Reidemeister} where the orientation of each strand is arbitrary). It is easy to see that $$\langle L \rangle^{}_{\mathcal{C},X}=\theta^{-w(D)}_{X}\langle L \rangle^{fr}_{\mathcal{C},X}$$ is an isotopy invariant of oriented links as it is preserved under Reidemeister moves $\Omega 1$, $\Omega 2$ and $\Omega 3$ where $w(D)$ is the writhe of a link diagram $D$ of an oriented link $L$ and $\theta_X$ is twist coefficient of $X$ in category $\mathcal{C}$. Note that if $X$ is self-dual then $\langle \:\:\: \rangle^{fr}_{\mathcal{C},X}$ is an invariant of \emph{unoriented} framed links. Next theorem shows that link invariant $T_S$ obtained from the gYB-operator $R_{X,\mathcal{L}}$ via certain enhancement is the same as invariant $\langle \:\:\: \rangle^{}_{\mathcal{C},X}$ up to a normalization factor.

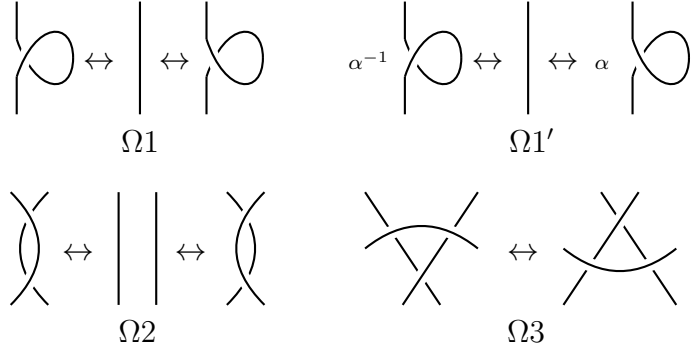
\begin{figure}[h]\psset{unit=5mm}

$\begin{pspicture}(0,0)(2.5,3)
\psline(0,0)(0,1)\psline(0,2)(0,3)
\psbezier(0,2)(0.5,0.5)(1.5,0.5)(1.5,1.5)\psline[linecolor=white,linewidth=5pt](0.1,1.2)(0.3,1.7)
\psbezier(0,1)(0.5,2.5)(1.5,2.5)(1.5,1.5)
%\put(-1.5,1.2){\tiny $ \alpha^{-1}$}
\end{pspicture}$
$\begin{pspicture}(0,0)(2,3)
\psline(0.5,0)(0.5,3)
\put(0,-1){$\Omega 1$}
\put(-1,1.2){$\leftrightarrow$}
\put(1,1.2){$\leftrightarrow$}
\end{pspicture}$
$\begin{pspicture}(0,0)(3,3)
\psline(0,0)(0,1)\psline(0,2)(0,3)
\psbezier(0,1)(0.5,2.5)(1.5,2.5)(1.5,1.5)\psline[linecolor=white,linewidth=5pt](0.3,1.2)(0.1,1.7)
\psbezier(0,2)(0.5,0.5)(1.5,0.5)(1.5,1.5)
%\put(-1.5,1.2){\tiny $ \alpha^{-1}$}
\end{pspicture}$
\hspace{1cm}$\begin{pspicture}(0,0)(2.5,3)
\psline(0,0)(0,1)\psline(0,2)(0,3)
\psbezier(0,2)(0.5,0.5)(1.5,0.5)(1.5,1.5)\psline[linecolor=white,linewidth=5pt](0.1,1.2)(0.3,1.7)
\psbezier(0,1)(0.5,2.5)(1.5,2.5)(1.5,1.5)
\put(-1.5,1.2){\tiny $ \alpha^{-1}$}
\end{pspicture}$
$\begin{pspicture}(0,0)(2,3)
\psline(0.5,0)(0.5,3)
\put(0,-1){$\Omega 1'$}
\put(-1,1.2){$\leftrightarrow$}
\put(1,1.2){$\leftrightarrow$}
\end{pspicture}$
\hspace{0.5cm}$\begin{pspicture}(0,0)(3,3)
\psline(0,0)(0,1)\psline(0,2)(0,3)
\psbezier(0,1)(0.5,2.5)(1.5,2.5)(1.5,1.5)\psline[linecolor=white,linewidth=5pt](0.3,1.2)(0.1,1.7)
\psbezier(0,2)(0.5,0.5)(1.5,0.5)(1.5,1.5)
\put(-1,1.2){\tiny $ \alpha$}
\end{pspicture}$

\vspace{1cm}
\hspace{-1cm}$\begin{pspicture}(0,0)(2,3)
\psbezier(1,3)(0,2)(0,1)(1,0)\psline[linecolor=white,linewidth=5pt](0.3,2.6)(0.6,2.2)\psline[linecolor=white,linewidth=5pt](0.3,0.3)(0.6,0.8)
\psbezier(0,3)(1,2)(1,1)(0,0)
\end{pspicture}$
\hspace{0.3cm}$\begin{pspicture}(0,0)(2,3)
\psline(0,0)(0,3)\psline(1,0)(1,3)
\put(0,-1){$\Omega 2$}
\put(-1.5,1.2){$\leftrightarrow$}
\put(1.5,1.2){$\leftrightarrow$}
\end{pspicture}$
\hspace{0.3cm}$\begin{pspicture}(0,0)(1,3)
\psbezier(0,3)(1,2)(1,1)(0,0)\psline[linecolor=white,linewidth=5pt](0.7,2.6)(0.4,2.2)\psline[linecolor=white,linewidth=5pt](0.7,0.3)(0.4,0.8)
\psbezier(1,3)(0,2)(0,1)(1,0)
\end{pspicture}$
\hspace{1.2cm}$\begin{pspicture}(0,0)(3,3)
\psline(0,3)(2,0)\psline[linecolor=white,linewidth=5pt](1,0)(3,3)\psline(1,0)(3,3)
\psbezier[linecolor=white,linewidth=5pt](0,1.5)(1,2.3)(2,2.3)(3,1.5)
\psbezier(0,1.5)(1,2.3)(2,2.3)(3,1.5)
\end{pspicture}$
\hspace{1cm}$\begin{pspicture}(0,0)(3,3)
\psline(1,3)(3,0)\psline[linecolor=white,linewidth=5pt](0,0)(2,3)\psline(0,0)(2,3)
\psbezier[linecolor=white,linewidth=5pt](0,1.5)(1,0.7)(2,0.7)(3,1.5)
\psbezier(0,1.5)(1,0.7)(2,0.7)(3,1.5)
\put(-1.5,-1){$\Omega 3$}
\put(-1.5,1.2){$\leftrightarrow$}
\end{pspicture}$

\vspace{1cm}\caption{\small Reidemeister moves }
\label{Reidemeister}
\end{figure}

\begin{thm}\label{equal}
Let $\mathcal{C}$ be a ribbon fusion category. Let $X \in \Irr (\mathcal{C})$ be a $(d,3,1)$-gYBE object with respect to a set $\mathcal{L}$ and $R_{X,\mathcal{L}}$ be the corresponding gYB-operator. Then $S=(R_{X,\mathcal{L}}, \mu=\Id_V, \alpha=\theta_X,\beta=1 )$ is an EgYB-operator and we have

$$\langle L \rangle^{}_{\mathcal{C},X}= d^{-1}\:\:T_S(L)$$

\noindent for any oriented link $L$.
\end{thm}

\begin{proof}
 Note that all quantum dimensions $\dim (X_i)$ for $X_i \in \mathcal{L}$ are the same. We denote it as $\varepsilon$. Also we denote $R_{X,\mathcal{L}}$ as simply $R$ for simplicity. Recall that $\dim(X)=d$ from Theorem \ref{KWthm}.

At first, let us prove the first statement. The commutativity condition $R\circ \mu^{\otimes 3} = \mu^{\otimes 3} \circ R$ is trivially hold as $\mu$ is the identity. For orthogonality condition, it is sufficient to show that $\Sp_{3,1}(R)-\theta_X \Id_{V^{\otimes 2}}$ acts off-diagonally on the second tensor factor (see Remark \ref{rmk-1}). That is, $\left( \Sp_{3,1}(R)-\theta_X \Id_{V^{\otimes 2}} \right)^{j_1,j_2}_{i_1,i_2}=0$ if $i_2=j_2$. (The other orthogonality can be verified in the same way.) We show equivalently that $\sum_{k \in I} R^{i,j,k}_{i,j,k}=\theta_X$ for any fixed $i, j \in I$. For this, we evaluate the following diagram in two different ways. Here we use the properties presented in equation (\ref{trivalent}) and Figure \ref{basis}. For simplicity we put $i \in I$ for $X_i \in \mathcal{L}$ and all other edges without any label attached should be read as edges labeled by $X$.

On one hand,

\vspace{-2cm}
$\begin{pspicture}(0,-1)(4,5)\psset{unit=5mm}
\psline(0,2.5)(0,3)
\psline(0,2.5)(1,2)
\psline(0,3)(1,4)
\psline(1,3)(1,4)
\psline(1,3)(2,2)
\psline[linecolor=white,linewidth=5pt](1.3,2.3)(1.7,2.7)
\psline(1,2)(2,3)
\psline(1,0)(1,2)
\psline(2,0)(2,2)
\psline(3,0)(3,3)
\psline(4,0)(4,4)
\psbezier(1,4)(1,5)(4,5)(4,4)
\psbezier(2,3)(2,3.7)(3,3.7)(3,3)
\psbezier(2,0)(2,-0.5)(3,-0.5)(3,0)
\psbezier(1,0)(1,-1)(4,-1)(4,0)

\put(-0.4,3){\tiny $ i$}
\put(0.5,1){\tiny$j$}
%\put(1.1,3){\tiny$X$}
%\put(2.1,3){\tiny$X$}
\end{pspicture}$
\hspace{-2cm}$\begin{array}{c} = \:\: \sum_{k\in I} \frac{1}{\sqrt{ d   }}    \\ \\ \\ \\ \\ \\ \\ \\\end{array}$
$\begin{pspicture}(0,-1)(4,5)\psset{unit=5mm}
\psline(0,2.5)(0,3)
\psline(0,2.5)(1,2)
\psline(0,3)(1,4)
\psline(1,3)(1,4)
\psline(1,3)(2,2)
\psline[linecolor=white,linewidth=5pt](1.3,2.3)(1.7,2.7)
\psline(1,2)(2,3)
\psline(1,1.5)(1,2)
\psline(2,1.5)(2,2)
\psline(1,1.5)(1.5,1.2)
\psline(2,1.5)(1.5,1.2)
\psline(1.5,0.3)(1.5,1.2)
\psline(1,0)(1.5,0.3)
\psline(2,0)(1.5,0.3)
\psline(3,0)(3,3)
\psline(4,0)(4,4)
\psbezier(1,4)(1,5)(4,5)(4,4)
\psbezier(2,3)(2,3.7)(3,3.7)(3,3)
\psbezier(2,0)(2,-0.5)(3,-0.5)(3,0)
\psbezier(1,0)(1,-1)(4,-1)(4,0)

\put(-0.4,3){\tiny$i$}
\put(0.5,1.5){\tiny$j$}
\put(0.5,0){\tiny$j$}
\put(1.6,0.5){\tiny$k$}
%\put(1.1,3){\tiny$X$}
%\put(2.1,3){\tiny$X$}
\end{pspicture}$
\hspace{-2cm}$\begin{array}{c} = \:\: \sum_{k\in I} \frac{1}{\sqrt{ d   }}   \\ \\ \\ \\ \\ \\ \\ \\\end{array}$
$\begin{pspicture}(0,-0.5)(4,5.5)\psset{unit=5mm}
\psline(0,3)(2,5)
\psline(0,3)(2,1)
\psline(1,4)(3,2)
\psline(3,2)(2,1)
\psline[linecolor=white,linewidth=5pt](1.5,2.5)(2.5,3.5)
\psline(1,2)(3,4)
\psline(3,4)(2,5)
\psline(4,5)(4,1)
\psbezier(2,5)(2,6)(4,6)(4,5)
\psbezier(2,1)(2,0)(4,0)(4,1)
\put(0.1,3.6){\tiny$ i$}
\put(0.9,1.3){\tiny$j$}
\put(0.9,4.4){\tiny$j$}
\put(1.6,0.4){\tiny$k$}
\put(1.6,5.4){\tiny$k$}
%\put(1.1,3.2){\tiny$X$}
%\put(2.5,3.2){\tiny$X$}
\end{pspicture}$

\vspace{-5cm}
\noindent $\begin{array}{c} = \:\: \sum_{k\in I} \frac{1}{\sqrt{ d   }} \sum_{j'\in I} R^{i,j',k}_{i,j,k}   \\ \\ \\ \\ \\ \\ \\ \\\end{array}$
$\begin{pspicture}(0,0)(4,6)\psset{unit=5mm}
\psline(0,3)(2,5)
\psline(0,3)(2,1)
\psline(1,4)(2,3)
\psline(2,3)(1,2)
\psline(2,5)(3,4)
\psline(2,1)(3,2)
\psline(3,4)(3,2)
\psline(4,5)(4,1)
\psbezier(2,5)(2,6)(4,6)(4,5)
\psbezier(2,1)(2,0)(4,0)(4,1)

\put(0,3.6){\tiny$ i$}
\put(0.8,1.3){\tiny$j'$}
\put(0.9,4.5){\tiny$j$}
\put(1.6,0.4){\tiny$k$}
\put(1.6,5.4){\tiny$k$}
%\put(1.1,3.2){\tiny$X$}
%\put(2.5,3.2){\tiny$X$}
\end{pspicture}$
\hspace{-2cm}$\begin{array}{c} = \:\: \sum_{k\in I} \frac{1}{\sqrt{ d   }}  R^{i,j,k}_{i,j,k} \sqrt{d^2 \varepsilon^2}\:\: = \varepsilon\sqrt{d} \: \sum_{k\in I}  R^{i,j,k}_{i,j,k}\\ \\  \\ \\ \\ \\ \\ \\\end{array}$

On the other hand, we can resolve the twist at first with a factor $\theta_X$ multiplied and then use the properties of trivalent basis to evaluate the resulting diagram. We obtain $\theta_X \varepsilon \sqrt{ d}$ in this way. By comparing these two evaluations, we obtain the desired equality.

Now let us prove the second statement. Note that both sides have the same normalization factor $\theta^{-w(D)}_X=\alpha^{-w(D)}$. So we will remove the factors and prove the equality for oriented \emph{framed} links. Let $L$ be an oriented \emph{framed} links and $\beta \in B_n$ be a braid whose closure is isotopic to $L$ as an oriented framed link. We denote the braid diagram whose all edges are labeled by $X$ as $\beta_X\in \Hom(X^{\otimes n},X^{\otimes n})$ and its closure as $\overline{\beta_X} \in \Hom(\mathbf{1},\mathbf{1})=\mathbb{C}$. Note that if we evaluate $\overline{\beta_X}$, the value is $\langle L \rangle^{fr}_{\mathcal{C},X}$. In the following we show that $\langle L \rangle^{fr}_{\mathcal{C},X}= d^{-1}\:\: \tr(\rho^R_n(\beta))$. For this we consider a diagram which contains $\overline{\beta_X}$ inside a loop labeled by $X_i \in \mathcal{L}$ and evaluate it in two ways.

\vspace{-5cm}
\noindent$\begin{array}{c} \sum_{i\in I}  \\ \\ \\ \\ \\ \\ \\ \\ \\ \\ \\ \\\end{array}$
$\begin{pspicture}(0,-1)(7,8)\psset{unit=5mm}
\psframe(0.5,4)(3.5,5)
\psline(0,2)(0,5)
\psline(1,2)(1,4)
\psline(2,2)(2,4)
\psline(3,2)(3,4)
\psline(4,2)(4,5)
\psline(5,2)(5,5)
\psline(6,2)(6,5)
\psline(7,2)(7,5)
\psbezier(0,5)(0,8)(7,8)(7,5)
\psbezier(1,5)(1,7)(6,7)(6,5)
\psbezier(2,5)(2,6)(5,6)(5,5)
\psbezier(3,5)(3,5.5)(4,5.5)(4,5)
\psbezier(0,2)(0,-1)(7,-1)(7,2)
\psbezier(1,2)(1,0)(6,0)(6,2)
\psbezier(2,2)(2,1)(5,1)(5,2)
\psbezier(3,2)(3,1.5)(4,1.5)(4,2)
\put(-0.3,5){\tiny$ i$}
\put(1.5,4.4){\tiny$\beta_X$}
\put(2.2,5){\tiny$\cdots$}
\put(2.2,3){\tiny$\cdots$}
\end{pspicture}$
\hspace{-3cm}$\begin{array}{c} = \sum_{i,j_1\in I} \frac{1}{\sqrt{ d    } } \:\:\:  \\ \\ \\ \\ \\ \\ \\ \\ \\ \\ \\ \\\end{array}$
$\begin{pspicture}(0,-1)(7,8)\psset{unit=5mm}
\psframe(0.5,4)(3.5,5)
\psline(0,4)(0,5)
\psline(0,4)(0.5,3.5)\psline(1,4)(0.5,3.5)\psline(0.5,2.5)(0.5,3.5)\psline(0,2)(0.5,2.5)\psline(1,2)(0.5,2.5)
\psline(2,2)(2,4)
\psline(3,2)(3,4)
\psline(4,2)(4,5)
\psline(5,2)(5,5)
\psline(6,2)(6,5)
\psline(7,2)(7,5)
\psbezier(0,5)(0,8)(7,8)(7,5)
\psbezier(1,5)(1,7)(6,7)(6,5)
\psbezier(2,5)(2,6)(5,6)(5,5)
\psbezier(3,5)(3,5.5)(4,5.5)(4,5)
\psbezier(0,2)(0,-1)(7,-1)(7,2)
\psbezier(1,2)(1,0)(6,0)(6,2)
\psbezier(2,2)(2,1)(5,1)(5,2)
\psbezier(3,2)(3,1.5)(4,1.5)(4,2)

\put(-0.3,5){\tiny$ i$}
\put(-0.3,1){\tiny$ i$}
\put(0,3){\tiny$ j_1$}
\put(1.5,4.4){\tiny$\beta_X$}
\put(2.2,5){\tiny$\cdots$}
\put(2.2,3){\tiny$\cdots$}
\end{pspicture}$

\vspace{-8cm}
\noindent$\begin{array}{c} = \sum_{i,j_1,j_2\in I} \left(\frac{1}{\sqrt{ d    } }\right)^2 \:\:\:  \\ \\ \\ \\ \\ \\ \\ \\ \\ \\ \\\end{array}$
$\begin{pspicture}(0,-1)(7,8)\psset{unit=5mm}
\psframe(0.5,4)(3.5,5)
\psline(0,4)(0,5)
\psline(0,4)(1,3.5)\psline(1,4)(0.5,3.75)\psline(2,4)(1,3.5)\psline(1,2.5)(1,3.5)\psline(0,2)(1,2.5)\psline(1,2)(0.5,2.25)\psline(2,2)(1,2.5)
\psline(3,2)(3,4)
\psline(4,2)(4,5)
\psline(5,2)(5,5)
\psline(6,2)(6,5)
\psline(7,2)(7,5)
\psbezier(0,5)(0,8)(7,8)(7,5)
\psbezier(1,5)(1,7)(6,7)(6,5)
\psbezier(2,5)(2,6)(5,6)(5,5)
\psbezier(3,5)(3,5.5)(4,5.5)(4,5)
\psbezier(0,2)(0,-1)(7,-1)(7,2)
\psbezier(1,2)(1,0)(6,0)(6,2)
\psbezier(2,2)(2,1)(5,1)(5,2)
\psbezier(3,2)(3,1.5)(4,1.5)(4,2)
\put(-0.3,5){\tiny$ i$}
\put(-0.3,1){\tiny$ i$}
\put(0.3,3.3){\tiny$ j_1$}
\put(0.3,2.5){\tiny$ j_1$}
\put(1,2.9){\tiny$ j_2$}
\put(1.5,4.4){\tiny$\beta_X$}
\put(2.2,5){\tiny$\cdots$}
\put(2.2,3){\tiny$\cdots$}
\end{pspicture}$
\hspace{-3cm}$\begin{array}{c} = \cdots =  \\ \\ \\ \\ \\ \\ \\ \\ \\ \\ \\ \end{array}$

\vspace{-7.5cm}
\noindent$\begin{array}{c} = \sum_{i,j_1,\ldots ,j_n\in I} \left(\frac{1}{\sqrt{ d    } }\right)^n \: \\ \\ \\ \\ \\ \\ \\ \\ \\ \\ \\ \end{array}$
$\begin{pspicture}(0,-1)(7,8)\psset{unit=5mm}
\psframe(0.5,4)(3.5,5)
\psline(0,4)(0,5)
\psline(0,4)(2,3.5)\psline(1,4)(0.6,3.85)\psline(2,4)(1.2,3.7)\psline(3,4)(2,3.5)\psline(2,2.5)(2,3.5)
\psline(0,2)(2,2.5)\psline(1,2)(0.6,2.15)\psline(2,2)(1.2,2.3)\psline(3,2)(2,2.5)
\psline(4,2)(4,5)
\psline(5,2)(5,5)
\psline(6,2)(6,5)
\psline(7,2)(7,5)
\psbezier(0,5)(0,8)(7,8)(7,5)
\psbezier(1,5)(1,7)(6,7)(6,5)
\psbezier(2,5)(2,6)(5,6)(5,5)
\psbezier(3,5)(3,5.5)(4,5.5)(4,5)
\psbezier(0,2)(0,-1)(7,-1)(7,2)
\psbezier(1,2)(1,0)(6,0)(6,2)
\psbezier(2,2)(2,1)(5,1)(5,2)
\psbezier(3,2)(3,1.5)(4,1.5)(4,2)
\put(-0.3,5){\tiny$ i$}
\put(-0.3,1){\tiny$ i$}
\put(0.4,3.5){\tiny$ j_1$}
\put(0.4,2.3){\tiny$ j_1$}
\put(2,3){\tiny$ j_n$}
\put(1.5,4.4){\tiny$\beta_X$}
\put(2.2,5){\tiny$\cdots$}
\put(2,3.7){\tiny$\cdots$}
\put(2.2,1.8){\tiny$\cdots$}
\end{pspicture}$
\hspace{-3.4cm}$\begin{array}{c} = \sum_{i,j_1,\ldots ,j_n\in I} \left(\frac{1}{\sqrt{ d    } }\right)^n \:  \\ \\ \\ \\ \\ \\ \\ \\ \\ \\ \\ \\\end{array}$
$\begin{pspicture}(0,-1)(6,8)\psset{unit=5mm}
\psline(0,4)(3,7)
\psline(1,5)(2,4)
\psline(2,6)(4,4)
\psline(3,7)(6,4)
\psline(0,4)(3,1)
\psline(1,3)(2,4)
\psline(2,2)(4,4)
\psline(3,1)(6,4)
\psline(7,2)(7,6)
\psbezier(3,7)(3,8)(7,8)(7,6)
\psbezier(3,1)(3,0)(7,0)(7,2)
\psframe[fillcolor=white,fillstyle=solid](1,3.5)(6,4.5)
\put(3,3.9){\tiny$ \beta_X$}
\put(0.2,4.5){\tiny$ i$}
\put(0.8,5.4){\tiny$ j_1$}
\put(0.9,2.3){\tiny$ j_1$}
\put(2.5,0.5){\tiny$ j_n$}
\put(2.5,7.5){\tiny$ j_n$}
\put(3,5.4){\tiny$ \cdots$}
\put(3,2){\tiny$ \cdots$}
\end{pspicture}$

\vspace{-10cm}$\begin{array}{c} = \sum_{i,j_1,\ldots ,j_n,j'_1,\ldots,j'_{n-1}\in I} \left(\frac{1}{\sqrt{ d    } }\right)^n
\left(  \rho^{R}_{n}(\beta)  \right)^{i,j'_1,\ldots, j'_{n-1},j_n}_{i,j_1,\ldots,j_{n-1},j_n}\:\:\:  \\ \\ \\ \\ \\ \\ \\ \\ \\ \\ \\ \\\end{array}$
$\begin{pspicture}(0,0)(6,8)\psset{unit=5mm}
\psline(0,4)(3,7)
\psline(1,5)(2,4)
\psline(2,6)(4,4)
\psline(3,7)(6,4)
\psline(0,4)(3,1)
\psline(1,3)(2,4)
\psline(2,2)(4,4)
\psline(3,1)(6,4)
\psline(7,2)(7,6)
\psbezier(3,7)(3,8)(7,8)(7,6)
\psbezier(3,1)(3,0)(7,0)(7,2)
\put(0.2,4.5){\tiny$ i$}
\put(0.8,5.4){\tiny$ j_1$}
\put(0.8,2.3){\tiny$ j'_1$}
\put(2.5,0.5){\tiny$ j_n$}
\put(2.5,7.5){\tiny$ j_n$}
\put(3,5.4){\tiny$ \cdots$}
\put(3,2){\tiny$ \cdots$}
\end{pspicture}$

\vspace{-2cm}$=\left(\frac{1}{\sqrt{ d    } }\right)^n
\tr \left(  \rho^{R}_{n}(\beta)  \right) \sqrt{\varepsilon d^n \varepsilon}$
$=\varepsilon \cdot \tr \left(  \rho^{R}_{n}(\beta)  \right)$

\vspace{1cm}
On the other hand, we can evaluate the same diagram in the category $\mathcal{C}$ as $\sum_{i \in I} \dim(X_i)\overline{\beta_X} $ $=d \cdot \varepsilon \cdot \langle L \rangle^{fr}_{\mathcal{C},X}$. This implies the second statement.

\end{proof}

\section {Examples}\label{section-examples}

We consider unitary ribbon fusion categories $\mathcal{C}=SO(N)_2$ described in \cite{NR, KW}, where $N=2r+1, r \geq 1$. Simple objects are denoted as $$ \Irr(\mathcal{C})=\{ X_0, X_{2\lambda_1}, X_{\gamma^1}, X_{\gamma^2}, \ldots, X_{\gamma^r}, X_{\varepsilon}, X_{\varepsilon'} \} $$ in \cite{NR}. We will denote $X_0$ as $\mathbf{1}$, $X_{2\lambda_1}$ as $Z$, and $X_{\gamma^i}$ as $X_{i}$ following \cite{KW}. All fusion rules are as follows:

\begin{enumerate}
\item $X_{\varepsilon} \otimes X_{\varepsilon} \cong \mathbf{1} \oplus \oplus^{i=r}_{i=1}X_{i}$
\item $ X_{\varepsilon} \otimes X_{i}\cong X_{\varepsilon}\oplus X_{\varepsilon'}, 1\leq i \leq r$
\item $X_{\varepsilon} \otimes X_{\varepsilon'}\cong Z \oplus \oplus^{i=r}_{i=1}X_{i}$
\item $Z\otimes X_{\varepsilon}\cong X_{\varepsilon'}$
\item $Z \otimes Z \cong \mathbf{1}$
\item $Z \otimes X_i \cong X_i, 1\leq i \leq r$
\item $X_i \otimes X_i \cong \mathbf{1} \oplus Z \oplus X_{\min\{2i, 2r+1-2i\}}, 1\leq i \leq r$
\item $X_i \otimes X_j \cong X_{j-i} \oplus X_{\min\{i+j,2r+1-i-j\}}, i < j$
\end{enumerate}

\noindent It is observed that the simple objects $X_i, 1 \leq i \leq r,$ are $(2, 3, 1)$-gYBE objects with respect to the set $\mathcal{L}=\{X_{\varepsilon},X_{\varepsilon'}\}$ in \cite{KW}.

For explicit examples we consider $X_1$ with $\mathcal{L}=\{X_{\varepsilon},X_{\varepsilon'}\}$. For $N=3$ case, the gYB-operator $R_{X_1,\mathcal{L}}$ is given in \cite{KW}, which is the same as $R_{-1}(3)$ below. For the cases of $N=5$ and $7$, the gYB-operator $R_{X_1,\mathcal{L}}$ is the same as $-R_{+1}(N)$ below. To obtain this, we use the following data:

\begin{equation}\label{data}
\begin{array}{c} R^{X_1,X_1}_{\mathbf{1}}=e^{\pi i (N+1)/N}, R^{X_1,X_1}_{Z}=e^{\pi i /N}, R^{X_1,X_1}_{X_2}=e^{\pi i (N-1)/N}, \theta_{X_1}=e^{\pi i (N-1)/N}\\
F^{X_{\varepsilon},X_1,X_1}_{X_{\varepsilon}}=F^{X_{\varepsilon'},X_1,X_1}_{X_{\varepsilon'}}
=\frac{1}{\sqrt{2}}\left( \begin{smallmatrix} 1&1\\1&-1 \end{smallmatrix}\right), F^{X_{\varepsilon},X_1,X_1}_{X_{\varepsilon'}}=F^{X_{\varepsilon'},X_1,X_1}_{X_{\varepsilon}}
=\frac{1}{\sqrt{2}}\left( \begin{smallmatrix} 1&-1\\1&1 \end{smallmatrix}\right)
\end{array}
\end{equation}

\noindent The braidings above can be obtained from \cite{NR} by using the naturality of twist for all odd $N \geq 5$. To obtain the $F$-matrices we need to solve the pentagon equations, which is in general very tedious. The above $F$ matrices were obtained by doing so for $N=5$ and $N=7$, and we expect that they remain to be correct for all odd $N\geq 9$. Notice that Theorem \ref{main_5} below holds for all odd $N\geq 3$ though. The basis for each $\Hom$-space is ordered as listed in $\Irr (\mathcal{C})$.

$$R_{\nu}(N)= \begin{pmatrix}  \nu C&0&iS&0\\0&-iS&0&C\\iS&0&\nu C&0\\0&C&0&-iS \end{pmatrix} \oplus \begin{pmatrix}-iS&0&C&0\\0&\nu C&0&iS\\C&0&-iS&0\\0&iS&0&\nu C \end{pmatrix}$$

\noindent where $\nu \in \{ +1, -1 \}$, $C =\cos(\pi /N)$, $S=\sin(\pi /N)$, and $i=\sqrt{-1}$. It is straightforward to show that $R_{\nu}(N)$ is a gYB-operator for any $\nu \in \{+1, -1\}$, and thus so is $-R_{\nu}(N)$.

Note that the categories $SO(N)_2$ are self-dual for all $N$ and thus $\langle \:\:\:\:\: \rangle^{fr}_{\mathcal{C},X_1}$ gives rise to an isotopy invariant of \emph{unoriented} framed links. The next theorem shows that the invariants $T_S$ for all odd $N \geq 3$ are specializations of the Kauffman polynomial invariant $F$. (For the uniqueness of the invariant, see Section 4.3 in \cite{Tu}.)

\begin{thm}\label{main_5}
Let $S=(R_{X_1,\mathcal{L}}, \mu=\Id_V, \alpha=\theta_{X_1},\beta=1 )$ for all odd $N \geq 3$. Then for any diagram $D$ of an unoriented link $L$, $\widetilde{T_S}(D)=\alpha^{w(D)}T_S(L)$ satisfies

$$\widetilde{T_S}(D_+)-\widetilde{T_S}(D_-)= \eta\cdot 2i \sin(\pi /N)\left(\widetilde{T_S}(D_0)-\widetilde{T_S}(D_{\infty}) \right)$$

\noindent where $\eta=-1$ for $N=3$ and $\eta=+1$ for all odd $N\geq 5$. The link diagrams $D_+$, $D_-$, $D_0$, and $D_{\infty}$ are identical except in a small disk where they look as in Figure \ref{D+-}.
\end{thm}

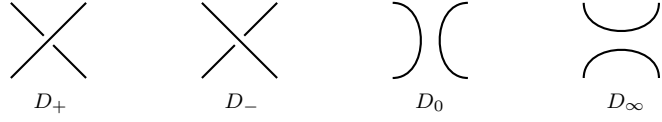
\begin{figure}[h]\psset{unit=7mm}

$\begin{pspicture}(0,0)(2,2)\psset{unit=5mm}
\psline(0,2)(2,0)
\psline[linecolor=white,linewidth=5pt](0.5,0.5)(1.5,1.5)
\psline(0,0)(2,2)
\put(0.6,-0.8){\tiny$ D_{+}$}
\end{pspicture}$
\hspace{1cm}$\begin{pspicture}(0,0)(2,2)\psset{unit=5mm}
\psline(0,0)(2,2)
\psline[linecolor=white,linewidth=5pt](0.5,1.5)(1.5,0.5)
\psline(0,2)(2,0)
\put(0.6,-0.8){\tiny$ D_{-}$}
\end{pspicture}$
\hspace{1cm}$\begin{pspicture}(0,0)(2,2)\psset{unit=5mm}
\psbezier(0,2)(1,2)(1,0)(0,0)
\psbezier(2,2)(1,2)(1,0)(2,0)
\put(0.6,-0.8){\tiny$ D_{0}$}
\end{pspicture}$
\hspace{1cm}$\begin{pspicture}(0,0)(2,2)\psset{unit=5mm}
\psbezier(0,0)(0,1)(2,1)(2,0)
\psbezier(0,2)(0,1)(2,1)(2,2)
\put(0.6,-0.8){\tiny$ D_{\infty}$}
\end{pspicture}$

\vspace{1cm}
\caption{Link diagrams $D_+$, $D_-$, $D_0$, and $D_{\infty}$. We denote corresponding unoriented framed links as $L_+$, $L_-$, $L_0$, and $L_{\infty}$.}\label{D+-}
\end{figure}

\begin{proof}
We prove the theorem for the cases of $N\geq 5$. Note that $\widetilde{T_S}(D)=\langle L \rangle^{fr}_{\mathcal{C},X_1}$ for a diagram $D$ of an unoriented framed link $L$, and thus it suffices to show that

\begin{equation}\label{fr_equality}\langle L_+ \rangle^{fr}_{\mathcal{C},X_1}-\langle L_- \rangle^{fr}_{\mathcal{C},X_1}= 2i \sin(\pi /N)\left( \langle L_0  \rangle^{fr}_{\mathcal{C},X_1}-\langle L_{\infty}\rangle^{fr}_{\mathcal{C},X_1}\right).
\end{equation}

 Since $\Hom(X^{\otimes 2}_{1},X^{\otimes 2}_{1})= \oplus_{Y\in \{\mathbf{1}, Z, X_2\} }  \Hom(Y, X^{\otimes 2}_{1})\otimes \Hom(X^{\otimes 2}_{1},Y)   $, we can express the braidings $c^{\pm}_{X_1,X_1}\in \Hom(X^{\otimes 2}_{1},X^{\otimes 2}_{1})$ as linear combinations of basis morphisms $\{f_{\mathbf{1}}, f_{Z}, f_{X_2}\}$ where $f_{Y} \in \Hom(Y, X^{\otimes 2}_{1})\otimes \Hom(X^{\otimes 2}_{1},Y)$ denotes the basis morphism for each $Y \in \{\mathbf{1}, Z, X_2\} $ (see Figure \ref{basis_morphisms}). From the properties of trivalent basis (see Figure \ref{basis}) and braiding data in equation (\ref{data}), we obtain

$$c^{\pm}_{X_1,X_1}= \left(-\frac{1}{2}e^{\pm \pi i/N}\right)f_{\mathbf{1}}+\left(\frac{1}{2}e^{\pm \pi i/N}\right)f_{Z} -\left(\frac{1}{\sqrt{2}}e^{\mp \pi i/N}\right)f_{X_2}   .$$

\noindent Now then

$c^{+}_{X_1,X_1}-c^{-}_{X_1,X_1}=\left( -i \sin (\pi /N) \right) f_{\mathbf{1}} +   \left( i \sin (\pi /N) \right) \left( f_{Z}  + \sqrt{2} f_{X_2} \right)\\ =\left( -i \sin (\pi /N) \right) f_{\mathbf{1}} +\left( i \sin (\pi /N) \right) \left( 2\Id_{X^{\otimes 2}_{1}} -f_{\mathbf{1}} \right)=2i \sin (\pi /N)\left( \Id_{X^{\otimes 2}_{1}}-f_{\mathbf{1}}  \right)$

\noindent This implies (\ref{fr_equality}).

The case of $N=3$ can be proved in the same way.

\end{proof}

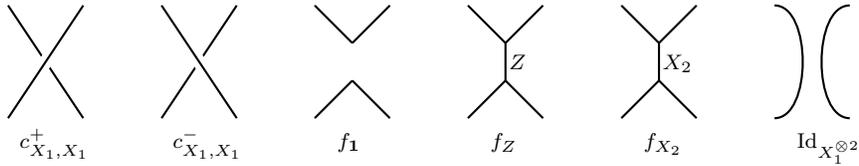
\begin{figure}[h]\psset{unit=7mm}

$\begin{pspicture}(0,0)(2,3)\psset{unit=5mm}
\psline(0,3)(2,0)
\psline[linecolor=white,linewidth=5pt](0.5,0.75)(1.5,2.25)
\psline(0,0)(2,3)
\put(0.3,-0.8){\tiny$ c^{+}_{X_1,X_1}$}
\end{pspicture}$
\hspace{0.5cm}$\begin{pspicture}(0,0)(2,3)\psset{unit=5mm}
\psline(0,0)(2,3)
\psline[linecolor=white,linewidth=5pt](0.5,2.25)(1.5,0.75)
\psline(0,3)(2,0)
\put(0.3,-0.8){\tiny$ c^{-}_{X_1,X_1}$}
\end{pspicture}$
\hspace{0.5cm}$\begin{pspicture}(0,0)(2,3)\psset{unit=5mm}
\psline(0,0)(1,1)\psline(2,0)(1,1)\psline(0,3)(1,2)\psline(2,3)(1,2)
\put(0.6,-0.8){\tiny$ f_{\mathbf{1}}$}
\end{pspicture}$
\hspace{0.5cm}$\begin{pspicture}(0,0)(2,3)\psset{unit=5mm}
\psline(0,0)(1,1)\psline(2,0)(1,1)\psline(0,3)(1,2)\psline(2,3)(1,2)\psline(1,1)(1,2)
\put(0.6,-0.8){\tiny$ f_{Z}$}\put(1.1,1.3){\tiny$ Z$}
\end{pspicture}$
\hspace{0.5cm}$\begin{pspicture}(0,0)(2,3)\psset{unit=5mm}
\psline(0,0)(1,1)\psline(2,0)(1,1)\psline(0,3)(1,2)\psline(2,3)(1,2)\psline(1,1)(1,2)
\put(0.6,-0.8){\tiny$ f_{X_2}$}\put(1.1,1.3){\tiny$ X_2$}
\end{pspicture}$
\hspace{0.5cm}$\begin{pspicture}(0,0)(2,3)\psset{unit=5mm}
\psbezier(0,0)(1,0)(1,3)(0,3)\psbezier(2,0)(1,0)(1,3)(2,3)
\put(0.6,-0.8){\tiny$ \Id_{X^{\otimes 2}_{1}}$}
\end{pspicture}$

\vspace{1cm}
\caption{Morphisms in $\Hom_{\mathcal{C}}(X^{\otimes 2}_{1},X^{\otimes 2}_{1})$. Morphisms $ c^{+}_{X_1,X_1}$, $ c^{-}_{X_1,X_1}$, $f_{\mathbf{1}}$, $\Id_{X^{\otimes 2}_{1}}$ correspond to $L_{+}$, $L_{-}$, $L_{\infty}$, $L_{0}$, respectively. }\label{basis_morphisms}
\end{figure}

\begin{rmk}
\begin{enumerate}
\item The Kauffman polynomial invariant $F$ is denoted as $Q_{\nu}$ in \cite{Tu} for $\nu \in \{1, -1\}$. If we normalize invariant $T_S$ in Theorem \ref{main_5} as $\hat{T}_S=\frac{1}{4}T_S$, it is multiplicative on a disjoint union of links and has value $1$ on the trivial knot. Indeed the invariants $\hat{T}_S$ are specializations of $Q_{-}$ with $y=\pm 2i\sin(\pi/N)$.
\item One can prove the equality (\ref{fr_equality}) from the minimal polynomial $R^3+e^{-\pi i/N}R^2-e^{2\pi i/N}R-e^{\pi i/N}\Id=0$ where $R=-R_{+1}(N)$.

\item We may consider more generally $R_{\nu}(\theta)$ using $\cos \theta$ and $\sin \theta$ in the definition of $R_{\nu}(N)$ for any value $0 \leq \theta < 2\pi$. It is easy to show that $R_{\nu}(\theta)$ is a $(2,3,1)$-gYB-operator as well. We expect that Theorem \ref{main_5} remains true for any $\theta$. For $\pm R_{\nu}(\theta)$, the factor on the right hand side would be $\mp 2i\sin(\theta)$ for both of $\nu \in \{1, -1\}$.

\end{enumerate}
\end{rmk}

\end{document}